\documentclass[12pt]{article}
\usepackage{amsfonts}
\usepackage{amssymb}
\usepackage{amssymb,amsfonts}
\usepackage{epsfig}
\usepackage{CJK}
\usepackage{color}
\usepackage{mathrsfs}
\usepackage{amsthm}
\usepackage{amsmath}
\usepackage{graphicx}
\allowdisplaybreaks

 \newtheorem{thm}{Theorem}[section]
 
 \newtheorem{lem}[thm]{Lemma}
 
 \theoremstyle{definition}
 
 \newtheorem{rem}[thm]{Remark}
 \numberwithin{equation}{section}

\def\ep{\epsilon}

\newtheorem{Lemma A.1}{Lemma A.1}
\theoremstyle{definition}

\theoremstyle{remark}

\begin{document}
\title{A Liouville type theorem for axially symmetric $D$-solutions to steady Navier-Stokes Equations}
\author{Na Zhao\footnotemark[1]}
\renewcommand{\thefootnote}{\fnsymbol{footnote}}

\footnotetext[1]{ Email: nzhao13@fudan.edu.cn. School of Mathematical Sciences, Fudan University,
Shanghai,  200433,  China.
Department of mathematics, University of California, Riverside, CA 92521, USA.}

\date{}

\maketitle

\begin{abstract}
We study axially symmetric $D$-solutions of three dimensional steady Navier-Stokes equations. We prove that
if the velocity $u$ decays like $|x'|^{-(\frac{2}{3})^+}$ uniformly for $z$, or the vorticity $\omega$ decays like $|x'|^{-(\frac{5}{3})^+}$ uniformly for $z$, then $u$ vanishes. Here $|x'|$ denotes the distance to the axis.
\end{abstract}

\section{Introduction}
We consider axially symmetric solutions to the three dimensional steady Navier-Stokes equations
\begin{equation}\label{a01}
\begin{cases}
 \Delta u=u\cdot \nabla u+ \nabla p\ \ \text{in}\  \mathbb{R}^3,\\
 \nabla \cdot u=0\ \ \text{in}\  \mathbb{R}^3,
\end{cases}
\end{equation}
with homogeneous condition at infinity
\begin{align}\label{a02}
\lim_{|x|\rightarrow \infty}u(x)=0,
\end{align}
and finite Dirichlet integral 
\begin{align}\label{a03}
\int |\nabla u|^2 dx<\infty.
\end{align}
Here $u$ and $p$ denote the velocity and pressure, respectively. Solutions satisfying \eqref{a01} with \eqref{a02} and \eqref{a03} are known as $D$-solutions. In 1933, Leray \cite{Le} first studied these solutions by variational method. Since then, an open question whether $D$-solutions equal 0 has attracted many mathematicians. However, there are only partial results obtained under some extra integral or decay assumptions on $u$. For instance, Galdi \cite{Ga} Theorem X.9.5 proved that if $u\in L^{\frac{9}{2}}(\mathbb{R}^3)$, then $u\equiv 0$. In 2016, Chae and Wolf \cite{CW} improved this result by a log factor. In \cite{Chae}, Chae proved that $D$-solutions are 0 if $\Delta u\in L^{\frac{6}{5}}(\mathbb{R}^3)$. The quantity $\|\Delta u\|_{L^{\frac{6}{5}}}$ has the same scaling with the known quantity $\|\nabla u\|_{L^{2}}$. Seregin \cite{S2} showed that $u\in L^6(\mathbb{R}^3)\cap BMO^{-1}$ also implies Liouville type theorem. On the other hand, authors in \cite{CPZ} considered $D$-solutions in the slab $\mathbb{R}^2\times [-\pi,\pi]$ with suitable boundary conditions and proved $u\equiv 0$. 

Recently, in \cite{KTW}, Kozono etc showed that if the vorticity $\omega$ decays faster than $|x|^{-\frac{5}{3}}$ at infinity, then $D$-solutions in $\mathbb{R}^3$ are 0. However, to the best of our knowledge, there is no known result for the \emph{a priori} decay rate (at infinity) of a general $D$-solution or vorticity in dimension three. This situation is different with axially symmetric case. More precisely, in axially symmetric case, even though we still don't know whether the Liouville type theorem holds, we have the \emph{a priori} decay estimate for velocity $u$ and vorticity $\omega$.  In \cite{CJ} and \cite{W}, the authors showed that the \emph{a priori} decay rate is
\begin{align*}
|u(x)|\leq C\big(\frac{\ln r}{r}\big)^{\frac{1}{2}},\ \ |\omega_\theta(x)|\leq C r^{-(\frac{19}{16})^-},\ \ 
|\omega_r(x)|+|\omega_z(x)|\leq C r^{-(\frac{67}{64})^-}.
\end{align*}
Here $C$ denotes a positive constant whose value may change from line to line. For a positive number $a$, we denote $a^-$ as a number smaller but close to $a$ and $a^+$ a number bigger but close to $a$. 
Recently, in \cite{CPZ}, authors gave a simple proof for the decay rate of  $u$ by using Brezis-Gallouet inequality \cite{BG} and improved the decay of $\omega$ to
\begin{align*}
 |\omega_\theta(x)|\leq C\frac{(\ln r)^{\frac{3}{4}}}{r^{\frac{5}{4}}} ,\ \ 
|\omega_r(x)|+|\omega_z(x)|\leq C \frac{(\ln r)^{\frac{11}{8}}}{r^{\frac{9}{8}}}.
\end{align*}
These results are ``what we know" for axially symmetric $D$-solutions. In this paper, we will show ``what we need" to prove Liouville type theorem. Then we will know how far it is from ``what we know" to ``what we need". More specifically, we prove that if velocity $u$ decays like 
\begin{align*}
|u(x)|\leq Cr^{-(\frac{2}{3})^+}
\end{align*}
or vorticity $\omega$ decays like
\begin{align*}
|\omega(x)|\leq r^{-(\frac{5}{3})^+},
\end{align*}
then $u\equiv 0$. We emphasize that this result doesn't need the decay on $z$-direction. Also we remark that we finished the result in April, 2018. Afterwards, we saw the paper by Wang Wendong \cite{Wang} in arXiv, who proved a similar result independently.

Throughout this paper, we will use $x=(x_1,x_2,x_3)$ or $x=(x_1,x_2,z)$ to denote a point in $\mathbb{R}^3$. Also, $x'=(x_1,x_2)$ and $r=|x'|=(x_1^2+x_2^2)^{\frac12}$. $A\lesssim B$ denotes there exists a constant $C$ such that $A\leq CB$ holds. Now we are ready to state our main results.
\begin{thm}\label{thm1}
Let $(u,p)$ be an axially symmetric solution of \eqref{a01} with \eqref{a02} and \eqref{a03}. Suppose that $u$ satisfies 
\begin{align}\label{a04}
|u(x)| \leq \frac{C}{(1+|x'|)^{\alpha}},\ \frac23<\alpha<1,
\end{align}
uniformly for $z$, where $x'=(x_1,x_2)$. Then $u\equiv 0$.
\end{thm}
\begin{thm}\label{thm2}
Let $(u,p)$ be an axially symmetric solution of \eqref{a01} with \eqref{a02} and \eqref{a03}. 
Let $\omega=\nabla \times u$ be the vorticity. Suppose that $\omega$ satisfies 
\begin{align}\label{a05}
|\omega(x)| \leq \frac{C}{(1+|x'|)^{\alpha+1}},\ \frac23<\alpha<1,
\end{align}
uniformly for $z$, where $x'=(x_1,x_2)$. Then $u\equiv 0$.
\end{thm}
In the following sections, we will prove our main results. In section 2, the decay estimate for pressure $p$ is investigated. In section 3, we prove Theorem \ref{thm1}. In section 4, we prove Theorem \ref{thm2}.

\section{Decay estimate for $p$}
In this section, we study the decay estimate for pressure $p$. This will be an important ingredient for the proof of 
Theorem \ref{thm1}. In \cite{Ga} Theorem X.5.1, Galdi proved that there exists a constant $p_0\in \mathbb{R}$ such that $p(x)\rightarrow p_0$ as $x\rightarrow \infty$. We know from the equation that $(u,p-p_0)$ is also a solution to \eqref{a01} if $(u,p)$ is a solution. Thus, without loss of generality, we assume that $p_0=0$.
\begin{lem}\label{lem1}
Under the condition of Theorem \ref{thm1}, we have the following estimate for pressure $p$:
\begin{align}\label{b01}
|p(x)|\leq \frac{C\ln (e+|x'|)}{(1+|x'|)^{2\alpha}},\ \frac23<\alpha<1,
\end{align}
uniformly for $z$.
\end{lem} 
\begin{proof}
We will first derive an integral formula for $p$, that is 
\begin{align}\label{b02}
 p (x) = -\frac{1}{3}|u(x)|^2+P.V.\int \partial_{y_i} \partial_{y_j} \Gamma (x-y) (u_i u_j) (y) dy.
\end{align}
Actually, this formula can be found in many literatures, for example, \cite{S1}. For the completeness, we will give the 
proof here.

Note that $p$ satisfies the following equation:
\begin{align}\label{p_equ}
-\Delta p = \partial_i \partial_j (u_i u_j).
\end{align}
Fix $x \in \mathbb{R}^3$. Let $R \gg 1$ and $\phi \in C^{\infty}_c(B(x,2R))$ be the cut-off function satisfying 
\begin{equation}\label{b03}
\begin{cases}
\phi \equiv 1\  \text{in}\  B(x,R);\\
\phi \equiv 0\  \text{in}\  \mathbb{R}^3\setminus B(x,2R);\\
|\nabla \phi|\leq \frac{C}{R},\ |\nabla^2 \phi|\leq \frac{C}{R^2}.
\end{cases}
\end{equation}
Multiplying $\phi$ in both sides of \eqref{p_equ} and using 
\[
\Delta(\phi p) = \Delta p \phi +2\nabla \phi \cdot \nabla p + \Delta \phi p,
\]
we know that
\begin{align*}
-\Delta (\phi p) = \partial_i \partial_j (u_i u_j) \phi - 2\nabla \phi \cdot \nabla p - \Delta \phi p.
\end{align*}
Let $\Gamma (x) = \frac{1}{4\pi |x|}$ and do integration by parts, we see 
\begin{equation}\label{c2}
    \begin{split}
\phi p (x) =& \int \Gamma (x-y) \partial_{y_i} \partial_{y_j} (u_i u_j) (y) \phi (y) dy\\
            &- \int \Gamma (x-y) 2\nabla_y \phi (y) \cdot \nabla_y p (y) dy
               -\int \Gamma (x-y) \Delta_y \phi(y) p (y) dy\\
           =&\big( -\frac13|u(x)|^2+P.V.\int \partial_{y_i} \partial_{y_j} \Gamma (x-y) (u_i u_j) (y) \phi (y) dy \big)\\
           &+2\int \partial_{y_i}  \Gamma (x-y) (u_i u_j) (y) \partial_{y_j} \phi (y) dy\\
           & +\int  \Gamma (x-y) (u_i u_j) (y) \partial_{y_i} \partial_{y_j}\phi (y) dy\\
           &+\int \Gamma (x-y) \Delta_y \phi(y) p (y) dy+2\int \nabla_y \Gamma (x-y) \cdot \nabla_y \phi (y)  p (y) dy\\      
           =&I_1+\cdots+I_5.
   \end{split}
\end{equation}

Let's show the detail of the integration by parts. The key point here is $I_1$ contains a singular integral since 
\[
\big| \partial_{y_i} \partial_{y_j} \Gamma (x-y) \big|\leq \frac{C}{|x-y|^3}
\]
and thus is not integrable when $y$ is close to $x$. So we should be careful when doing IBP (short for integration by parts). Here we only show how to get $I_1$ by IBP. Let $B(R):=B(x,R)$. The first IBP is to move $\partial_{y_i}$ to $\Gamma$:
\begin{align}
   & \int \Gamma (x-y) \partial_{y_i} \partial_{y_j} (u_i u_j) (y) \phi (y) dy \notag\\
= \lim_{\epsilon \rightarrow 0}
   &\int_{B(2R)\setminus B(\epsilon)} \Gamma (x-y) \partial_{y_i} \partial_{y_j} (u_i u_j) (y) \phi (y) dy\notag\\
= \lim_{\epsilon \rightarrow 0}
   &\Big(-\int_{B(2R)\setminus B(\epsilon)} \partial_{y_i}\Gamma (x-y)  \partial_{y_j} (u_i u_j) (y) \phi (y) dy
      -\int_{B(2R)\setminus B(\epsilon)} \Gamma (x-y)  \partial_{y_j} (u_i u_j) (y) \partial_{y_i}\phi (y) dy\Big)
      \notag\\
   &+\int_{\partial B(2R)} \Gamma (x-y) \partial_{y_j} (u_i u_j) (y) n_i \phi (y) dS_y\label{c2_1}\\
   &- \lim_{\epsilon \rightarrow 0} \int_{\partial B(\epsilon)} \Gamma (x-y) \partial_{y_j} (u_i u_j) (y) n_i \phi (y) dS_y\label{c2_2}\\
 =&-\int \partial_{y_i}\Gamma (x-y)  \partial_{y_j} (u_i u_j) (y) \phi (y) dy
   -\int \Gamma (x-y)  \partial_{y_j} (u_i u_j) (y) \partial_{y_i}\phi (y) dy\notag,
\end{align}
where $n_i=\frac{y_i-x_i}{|x-y|}$ is the $i$-th component of the outer normal vector of $B(2R)$.
In the above, we have used the fact that \eqref{c2_1} equals $0$ because of $\phi=0$ on $\partial B(2R)$. In terms of \eqref{c2_2}, we have the following estimate:
\begin{align*}
 &\int_{\partial B(\epsilon)} \Gamma (x-y) \partial_{y_j} (u_i u_j) (y) n_i \phi (y) dS_y\\
\lesssim &\int_{\partial B(\epsilon)} \frac{1}{|x-y|} dS_y \max_{y\in \partial B(\epsilon)} |u(y)|^2\\
\lesssim &\frac{1}{\epsilon} \cdot \epsilon^2\max_{y\in \partial B(\epsilon)} |u(y)|^2\rightarrow 0\ as\ \epsilon \rightarrow 0.
\end{align*}
Next we act the second IBP after which the singular integral will come out.
\begin{align}
   \lim_{\epsilon \rightarrow 0}
   &-\int_{B(2R)\setminus B(\epsilon)} \partial_{y_i}\Gamma (x-y)  \partial_{y_j} (u_i u_j) (y) \phi (y) dy\notag\\
= \lim_{\epsilon \rightarrow 0}
   &\Big(\int_{B(2R)\setminus B(\epsilon)}\partial_{y_j} \partial_{y_i}\Gamma (x-y) (u_i u_j) (y) \phi (y) dy
   +\int_{B(2R)\setminus B(\epsilon)} \partial_{y_i}\Gamma (x-y) (u_i u_j) (y) \partial_{y_j}\phi (y) dy\Big)\notag\\
   &-\int_{\partial B(2R)}\partial_{y_i} \Gamma (x-y) n_j (u_i u_j) (y)  \phi (y) dS_y\notag\\
   & +\lim_{\epsilon \rightarrow 0} \int_{\partial B(\epsilon)} \partial_{y_i} \Gamma (x-y) n_j (u_i u_j) (y) 
    \phi (y)  dS_y\label{c2_3} \\
 =-\frac13&|u(x)|^2+P.V. \int \partial_{y_j} \partial_{y_i}\Gamma (x-y) (u_i u_j) (y) \phi (y) dy
 +\int \partial_{y_i}\Gamma (x-y) (u_i u_j) (y) \partial_{y_j}\phi (y) dy,\notag
\end{align}
where we have dealt with \eqref{c2_3} as follows.
\begin{align*}
 &\lim_{\epsilon \rightarrow 0}\int_{\partial B(\epsilon)} \partial_{y_i} \Gamma (x-y) n_j (u_i u_j) (y)
     \phi (y)  dS_y\\
=&\lim_{\epsilon \rightarrow 0}\int_{\partial B(\epsilon)} \frac{x_i-y_i}{4\pi |x-y|^3} 
      \frac{y_j-x_j}{|x-y|} (u_i u_j) (y)dS_y\\
=&-\lim_{\epsilon \rightarrow 0}\frac{1}{4\pi \epsilon^4} \int_{\partial B(\epsilon)} (x_i-y_i)(x_j-y_j)(u_i u_j) (y)dS_y\\
=& -\lim_{\epsilon \rightarrow 0}\frac{1}{4\pi \epsilon^4}\int_{\partial B(\epsilon)} (x_i-y_i)(x_j-y_j)
\big[(u_i u_j) (y)-(u_i u_j) (x)\big]dS_y\\
&-\lim_{\epsilon \rightarrow 0}\frac{(u_i u_j) (x)}{4\pi \epsilon^4} \int_{\partial B(\epsilon)} 
(x_i-y_i)(x_j-y_j)dS_y\\
=&J_1+J_2.
\end{align*}
For $J_1$, we know from the mean value inequality that
\begin{align*}
& \lim_{\epsilon \rightarrow 0}\frac{1}{4\pi \epsilon^4}\int_{\partial B(\epsilon)} (x_i-y_i)(x_j-y_j)
          \big[(u_i u_j) (y)-(u_i u_j) (x)\big]dS_y\\
\leq &\lim_{\epsilon \rightarrow 0} \frac{1}{4\pi \epsilon^4} \int_{\partial B(\epsilon)} (x_i-y_i)(x_j-y_j)
          \nabla (u_i u_j)(\xi)|x-y|dS_y\\
\leq & \lim_{\epsilon \rightarrow 0}\frac{1}{4\pi \epsilon^4} \epsilon^3\times 4\pi \epsilon^2  \max_{y\in \partial B(\epsilon)}|\nabla (u_i u_j)(y)|      =0.
\end{align*}
For $J_2$, we need to consider two cases: $i\neq j$ and $i=j$. 
When $i\neq j$, according to the symmetry of the function and the sphere, we have
\begin{align*}
    &\int_{\partial B(\epsilon)} (x_i-y_i)(x_j-y_j)dS_y\\
= &\int_{\begin{subarray}{l}\{y\in \partial B(\epsilon)\}\\ \{x_i-y_i>0\}\end{subarray}} (x_i-y_i)(x_j-y_j)dS_y
+\int_{\begin{subarray}{l}\{y\in \partial B(\epsilon)\}\\ \{x_i-y_i<0\}\end{subarray}} (x_i-y_i)(x_j-y_j)dS_y\\
=&0.
\end{align*}
When $i=j$, we have, for each $i$,
\begin{align*}
    \int_{\partial B(\epsilon)} (x_i-y_i)^2dS_y
= \frac13 \int_{\partial B(\epsilon)} |x-y|^2dS_y
=\frac{4\pi \epsilon^4}{3}.
\end{align*}
Summing $i=1,2,3$ gives
\begin{align*}
  \sum_{i=1}^3 \lim_{\epsilon \rightarrow 0}\frac{u_i^2  (x)}{4\pi \epsilon^4} \int_{\partial B(\epsilon)} 
(x_i-y_i)^2dS_y=\frac13 |u(x)|^2.
\end{align*}
Therefore,
\begin{align*}
J_2=-\lim_{\epsilon \rightarrow 0}\frac{(u_i u_j) (x)}{4\pi \epsilon^4} \int_{\partial B(\epsilon)} 
(x_i-y_i)(x_j-y_j)dS_y=-\frac13 |u(x)|^2.
\end{align*}

Let us then deal with $I_2$.
\begin{align*}
I_2 =& \int_{B(x,2R)-B(x,R)} \partial_{y_i}  \Gamma (x-y) (u_i u_j) (y) \partial_{y_j} \phi (y) dy\\
      \leq & \frac{C}{R^3} |B(x,2R)-B(x,R)| \max_{y\in B(x,2R)-B(x,R)} |u(y)|^2\\
      \leq & C\max_{y\in B(x,2R)-B(x,R)} |u(y)|^2.
\end{align*}
Since $u \rightarrow 0$ when $R \rightarrow \infty$, we know $I_2\rightarrow 0$ when $R \rightarrow \infty$.
Similarly, we can get $I_i\rightarrow 0$ for $i=3,4,5$ when $R \rightarrow \infty$. 
As a result, letting $R\rightarrow \infty$ in both sides of \eqref{c2} gives \eqref{b02}.

Now we will derive the decay rate of $p$ by \eqref{b02} under the assumption \eqref{a04} by using the method of 
Lemma 3.2 in \cite{CPZ}.
It is easy to check that $\partial_{y_j} \partial_{y_i} \Gamma (x-y)$ $(i,j=1,2,3)$ is a Calderon-Zygmud
kernel since it satisfies, for each $i,j$,
\begin{align}
&|\partial_{y_j} \partial_{y_i} \Gamma (x-y) |\leq \frac{C}{|x-y|^3},
\ |\nabla_y \partial_{y_j} \partial_{y_i} \Gamma (x-y) |\leq \frac{C}{|x-y|^4};\notag\\
&\int_{a<|x-y|<b}\partial_{y_j} \partial_{y_i} \Gamma (x-y) dy=0\ \text{for all} \ 0<a<b<\infty.\label{c3_2}
\end{align}
Also, we know from \eqref{c3_2} that
\begin{align*}
P.V. \int_{|x-y|\leq 1}\partial_{y_i} \partial_{y_j} \Gamma (x-y) dy=0.
\end{align*}

Now we decompose the integration in \eqref{b02} into three parts as follows.
\begin{align*}
 p (x) =& \int_{|x-y|\leq 1} \partial_{y_i} \partial_{y_j} \Gamma (x-y) \big[(u_i u_j) (y)- (u_i u_j) (x)\big]dy\\
 &+\int_{\{|x-y|\geq 1\}\cap \{|x'-y'|\geq \frac12\}} \partial_{y_i} \partial_{y_j} \Gamma (x-y) (u_i u_j) (y) dy\\
 &+\int_{\{|x-y|\geq 1\}\cap \{|x'-y'|\leq \frac12\}} \partial_{y_i} \partial_{y_j} \Gamma (x-y) (u_i u_j) (y) dy\\
 =&I_1+I_2+I_3.
\end{align*}
We next estimate $I_i$ $(i=1,2,3)$ one by one.
For $I_1$, by using the mean-value inequality, we have
\begin{align*}
I_1 \lesssim & \int_{\{|x-y|\leq 1\}} \big|\partial_{y_i} \partial_{y_j} \Gamma (x-y)\big| |x-y| |\nabla (u_i u_j)(\xi)|dy,\ |\xi| \in \big(\min\{|x|,|y|\},\max\{|x|,|y|\} \big)\\
\lesssim &  \int_{\{|x-y|\leq 1\}} \frac{1}{|x-y|^3} |x-y| |u(\xi)| |\nabla u(\xi)| dy\\
\lesssim & \int_{\{|x-y|\leq 1\}} \frac{1}{|x-y|^2}dy \frac{1}{(1+|\xi'|)^{\alpha}}
\frac{1}{(1+|\xi'|)^{\frac{67}{64}^{-}}}\\
\lesssim & \frac{1}{(1+|x'|)^{2\alpha}},
\end{align*} 
where we have used the assumption \eqref{a04} and the a priori estimate for $\nabla u$ in axially symmetric case.
Actually, from \cite{CPZ}, \cite{CJ} and \cite{W}, we know that
\begin{equation*}
\begin{split}
|\nabla u^r(x)|+|\nabla u^z(x)|\lesssim |x'|^{-\frac{5}{4}^{-}},\\
|\nabla u^{\theta}(x)|\lesssim |x'|^{-\frac{67}{64}^{-}}.
\end{split}
\end{equation*}

For $I_2$, we have
\begin{align*}
I_2 
\lesssim &\int_{\{|x-y|\geq 1\}\cap \{|x'-y'|\geq \frac12\}} \frac{1}{|x-y|^3} \frac{1}{(1+|y'|)^{2\alpha}}  dy\\
\lesssim & \int_{\{|x-y|\geq 1\}\cap \{|x'-y'|\geq \frac12\}}\frac{1}{(1+|y'|)^{2\alpha}} \int_{-\infty}^{\infty}\frac{1}{(|x'-y'|+|x_3-y_3|)^3} dy_3  dy'\\
\lesssim & \int_{\{|x'-y'|\geq \frac12\}}\frac{1}{(1+|y'|)^{2\alpha}} \frac{1}{|x'-y'|^2} dy'\\
=&\Big( \int_{\begin{subarray}{l}\{|x'-y'|\geq \frac12\}\\ \{|y'|\geq 2|x'|\}\end{subarray}}
   +\int_{\begin{subarray}{l}\{|x'-y'|\geq \frac12\}\\ \{\frac12 |x'| \leq |y'| \leq 2|x'|\}\end{subarray}}
   +\int_{\begin{subarray}{l}\{|x'-y'|\geq \frac12\}\\ \{|y'|\leq \frac12|x'|\}\end{subarray}}\Big)
   \frac{1}{(1+|y'|)^{2\alpha}} \frac{1}{|x'-y'|^2} dy'\\
 =& I_{21}+I_{22}+I_{23}.
\end{align*} 
For $I_{21}$, $|y'|\geq 2|x'|$ means that $|x'-y'| \approx |y'|$. Thus,
\begin{align*}
I_{21} 
\lesssim &\int_{ \{|y'|\geq 2|x'|\}} \frac{1}{(1+|y'|)^{2\alpha+2}} dy'\\
\lesssim &\int_{2|x'|}^{\infty} \frac{1}{(1+r)^{2\alpha+2}} r dr \  \text{(need $\alpha>0$)}\\
\lesssim & \frac{1}{(1+|x'|)^{2 \alpha}}.
\end{align*}
For $I_{22}$, $\frac12 |x'| \leq |y'| \leq 2|x'|$ means that $|x'-y'| \leq 3|x'|$. Thus,
\begin{align*}
I_{22} 
\lesssim &\int_{\{\frac12\leq |x'-y'| \leq 3|x'|\}}\frac{1}{|x'-y'|^2}dy'\frac{1}{(1+|x'|)^{2\alpha}} \\
\lesssim &\int_{\frac12}^{3|x'|} \frac{1}{r^2} r dr\frac{1}{(1+|x'|)^{2\alpha}}\\
\lesssim & \frac{\ln (e+|x'|)}{(1+|x'|)^{2\alpha}}.
\end{align*}
For $I_{23}$, $ |y'| \leq \frac12|x'|$ means that $|x'-y'| \approx |x'|$. Thus,
\begin{align*}
I_{23} 
\lesssim &\int_{\begin{subarray}{l}\{|x'-y'|\geq \frac12\}\\ \{|y'|\leq \frac12|x'|\}\end{subarray}}
   \frac{1}{(1+|y'|)^{2\alpha}}  dy' \frac{1}{(1+|x'|)^2}\\
   \lesssim & \int_{0}^{\frac12|x'|} \frac{1}{(1+r)^{2\alpha}} r dr \frac{1}{(1+|x'|)^2}\ \text{(need $\alpha<1$)}\\
\lesssim & \frac{1}{(1+|x'|)^{2\alpha}}.
\end{align*}
Combining the estimate of $I_{21}$, $I_{22}$ and $I_{23}$, we have
\begin{align*}
I_2 \lesssim  \frac{\ln (e+|x'|)}{(1+|x'|)^{2\alpha}}.
\end{align*}

For $I_3$, we have
\begin{align*}
I_3 
\lesssim & \int_{\begin{subarray}{l}\{|x-y|\geq 1\}\\ \{|x'-y'|\leq \frac12\}\end{subarray}} 
                  \frac{1}{|x-y|^3} \frac{1}{(1+|y'|)^{2\alpha}}  dy\\
\lesssim &  \int_{\begin{subarray}{l}\{|x_3-y_3|\geq \frac12\}\\ \{|x'-y'|\leq \frac12\}\end{subarray}} 
                  \frac{1}{(|x'-y'|+|x_3-y_3|)^3}  dy \frac{1}{(1+|x'|)^{2\alpha}}\\
 \lesssim &  \int_{\{|x_3-y_3|\geq \frac12\}} \int_{\{|x'-y'|\leq \frac12\}}
                  \frac{1}{(|x'-y'|+|x_3-y_3|)^3}  dy'dy_3 \frac{1}{(1+|x'|)^{2\alpha}}\\             
\lesssim & \frac{1}{(1+|x'|)^{2\alpha}}.
\end{align*} 

Combining the estimate of $I_1$, $I_2$ and $I_3$, we get the decay estimate \eqref{b01} for $p$.
\end{proof}
\section{Proof of Theorem \ref{thm1}}
In this section, we will prove Theorem \ref{thm1}.
\begin{proof}[Proof of Theorem \ref{thm1}]
Let $\eta(\xi)\in C_c^\infty(\mathbb{R})$ be the cut-off function satisfying
\begin{equation*}
\begin{cases}
\eta(\xi)=1,\ \text{if} \ |\xi|<1;\\
\eta(\xi)=0,\ \text{if} \ |\xi|>2;\\
0\leq \eta(\xi)\leq 1. 
\end{cases}
\end{equation*}
Define $\psi(x)=\eta(\frac{|x|}{R})$ where $R>1$ is a large number. Let $B_R:=B(0,R)$, then we have 
\begin{align*}
|\nabla \psi|\leq \frac{C}{R},\quad |\nabla^2 \psi| \leq \frac{C}{R^2};\\
\text{supp}\ \psi = B_{2R},\quad  \text{supp}\ \nabla \psi = B_{2R}\setminus B_{R}.
\end{align*}
 Multiplying
$u\psi$, integrating over $\mathbb{R}^3$ in \eqref{a01} and doing integration by parts, we have
\begin{align}\label{c5}
\begin{split}
\int |\nabla u|^2 \psi  dx
&= \frac12 \int |u|^2 \Delta \psi dx+ \int p u\cdot \nabla \psi dx+ \frac12 \int |u|^2 u\cdot \nabla \psi dx\\
&=K_1+K_2+K_3.
\end{split}
\end{align}
We are now going to show $K_i\rightarrow 0$ as $R\rightarrow \infty$ for $i=1,2,3$.  

First, for $K_1$, by H\"older inequality, we have
\begin{align*}
|K_1|
=&\Big| \frac12 \int_{B_{2R}\setminus B_R} |u|^2 \Delta \psi dx\Big| \\
\lesssim & \frac{1}{R^2}\Big(\int_{B_{2R}\setminus B_R} |u|^6 dx\Big)^{\frac13}\Big(\int_{B_{2R}\setminus B_R}  dx\Big)^{\frac23}\\
\lesssim &\Big(\int_{B_{2R}\setminus B_R} |u|^6 dx\Big)^{\frac13}.
\end{align*}
By Sobolev embedding, we know $u\in L^6(\mathbb{R}^3)$.
Thus $K_1\rightarrow 0$ when $R\rightarrow \infty$.

Next, in terms of $K_2$, one has
\begin{align*}
|K_2|
=&\Big| \int_{B_{2R}\setminus B_R} p u\cdot \nabla \psi dx\Big| \\
\leq & \int_{B_{2R}\setminus B_R}|p| |u| |\nabla \psi| dx\\
=&\int_{B_{2R}\setminus B_R} |p| |u|^{1-\epsilon} |u|^{\epsilon} |\nabla \psi| dx,
\end{align*} 
where $0<\epsilon\leq 1$ is a number to be chosen later.  
By H\"older inequality, we get
\begin{align*}
|K_2|
\leq &\int_{B_{2R}\setminus B_R} |p| |u|^{1-\epsilon} |u|^{\epsilon} |\nabla \psi| dx\\
\lesssim & \frac{1}{R}\Big(\int_{B_{2R}\setminus B_R} (|p| |u|^{1-\epsilon})^{\frac{6}{6-\ep}}  dx\Big)^{\frac{6-\epsilon}{6}}\Big(\int_{B_{2R}\setminus B_R} |u|^{\epsilon\cdot \frac{6}{\epsilon}} dx\Big)^{\frac{\epsilon}{6}}\\
\lesssim &\frac{1}{R}\Big(\int_{B_{2R}\setminus B_R} (|p| |u|^{1-\epsilon})^{\frac{6}{6-\ep}}   dx\Big)^{\frac{6-\epsilon}{6}}\Big(\int_{B_{2R}\setminus B_R} |u|^6 dx\Big)^{\frac{\epsilon}{6}} .
\end{align*}
If one can choose $\ep$ such that 
\begin{equation}\label{c6}
\widetilde{K_2}:=\frac{1}{R}\Big(\int_{B_{2R}\setminus B_R} (|p| |u|^{1-\epsilon})^{\frac{6}{6-\ep}}   dx\Big)^{\frac{6-\epsilon}{6}}\leq C
\end{equation}
holds, then $K_2\rightarrow 0$ as $R\rightarrow \infty$ due to the fact $u\in L^6(\mathbb{R}^3)$.
We claim that such $\ep$ do exist. Actually, from the assumption \eqref{a04} and Lemma \ref{lem1}, we have 
\begin{align*}
\widetilde{K_2}
\lesssim &\frac{1}{R}\Big(\int_{-2R}^{2R}\int_0^{2R} \Big[\frac{\ln (e+r)}{(1+r)^{2\alpha}}\cdot \frac{1}{(1+r)^{\alpha(1-\ep)}} \Big]^{\frac{6}{6-\ep}}  rdrdz\Big)^{\frac{6-\epsilon}{6}}\\
\lesssim &\frac{\ln R}{R}\Big(\int_{-2R}^{2R}\int_0^{2R}\frac{1}{(1+r)^{\alpha \frac{6(3-\ep)}{6-\ep}}}  rdrdz\Big)^{\frac{6-\epsilon}{6}}\\
\lesssim &R^{-\frac{\ep}{6}}\ln R \Big(\int_0^{2R} (1+r)^{1-\alpha  \frac{6(3-\ep)}{6-\ep}} dr\Big)^{\frac{6-\epsilon}{6}}.
\end{align*}
For any fixed $\alpha\in (\frac{2}{3},1)$, choose $\ep=\ep(\alpha)$ such that
\begin{align}\label{c7}
1-\alpha \frac{6(3-\ep)}{6-\ep}\leq -1,
\end{align}
then \eqref{c6} holds. In fact, if $1-\alpha \frac{6(3-\ep)}{6-\ep}= -1$, then
\[
\widetilde{K_2}\lesssim R^{-\frac{\ep}{6}} \ln R (\ln R)^{\frac{6-\epsilon}{6}}\leq C.
\]
If $1-\alpha \frac{6(3-\ep)}{6-\ep}< -1$, then
\begin{align*}
\widetilde{K_2}
\lesssim R^{-\frac{\ep}{6}}\ln R (R^{2-\alpha \frac{6(3-\ep)}{6-\ep}}+1)^{\frac{6-\epsilon}{6}}
\lesssim R^{-\frac{\ep}{6}}\ln R\leq C.
\end{align*}
Consequently, from \eqref{c7}, we can choose $\ep$ satisfying 
\begin{align}\label{c8}
0<\ep\leq \min \{1, \frac{3(3\alpha-2)}{3\alpha-1}\}
\end{align}
to guarantee \eqref{c6} and thus obtain $K_2\rightarrow 0$.

Finally, we deal with $K_3$ by using the same method with $K_2$. Indeed,
\begin{align*}
|K_3|
=&\Big| \frac12 \int_{B_{2R}\setminus B_R} |u|^2 u\cdot \nabla \psi dx\Big| \\
\leq & \frac12 \int_{B_{2R}\setminus B_R} |u|^3 |\nabla \psi| dx\\
=&\frac12 \int_{B_{2R}\setminus B_R} |u|^{\epsilon} |u|^{3-\epsilon} |\nabla \psi| dx.
\end{align*} 
Here we choose the same $\ep$ with $K_2$, i.e.,  $\ep=\ep(\alpha)$ satisfy \eqref{c8}.
Then by H\"older inequality, one has
\begin{align*}
|K_3|
\leq &\frac12 \int_{B_{2R}\setminus B_R} |u|^{\epsilon} |u|^{3-\epsilon} |\nabla \psi| dx\\
\lesssim & \frac{1}{R}\Big(\int_{B_{2R}\setminus B_R} |u|^{\epsilon\cdot \frac{6}{\epsilon}} dx\Big)^{\frac{\epsilon}{6}}
\Big(\int_{B_{2R}\setminus B_R} |u|^{(3-\epsilon)\cdot \frac{6}{6-\ep}}  dx\Big)^{\frac{6-\epsilon}{6}}\\
\lesssim &\frac{1}{R}\Big(\int_{B_{2R}\setminus B_R} |u|^6 dx\Big)^{\frac{\epsilon}{6}}
\Big(\int_{B_{2R}\setminus B_R} |u|^{(3-\epsilon)\cdot \frac{6}{6-\ep}}  dx\Big)^{\frac{6-\epsilon}{6}} .
\end{align*}
According to the assumption \eqref{a04}, one has
\begin{align*}
\widetilde{K_3}:=&\frac{1}{R}\Big(\int_{B_{2R}\setminus B_R} |u|^{(3-\epsilon)\cdot \frac{6}{6-\ep}}  dx\Big)^{\frac{6-\epsilon}{6}}\\
\lesssim &\frac{1}{R}\Big(\int_{-2R}^{2R}\int_0^{2R} \frac{1}{(1+r)^{\alpha \frac{6(3-\ep)}{6-\ep}}}  rdrdz\Big)^{\frac{6-\epsilon}{6}}\\
\lesssim &R^{-\frac{\ep}{6}}\Big(\int_0^{2R} (1+r)^{1-\alpha \frac{6(3-\ep)}{6-\ep}}  dr\Big)^{\frac{6-\epsilon}{6}}.
\end{align*}
Similarly to $\widetilde{K_2}$, one can show $\widetilde{K_3}\leq C$ under \eqref{c8} and thus obtain $K_3\rightarrow 0$ as $R\rightarrow \infty$.

Let $R\rightarrow \infty$ in both sides of \eqref{c5}, we get
\begin{align*}
\int_{\mathbb{R}^3} |\nabla u|^2=0,
\end{align*}
which gives $u\equiv 0$.

\end{proof}

\section{Proof of Theorem \ref{thm2}}
In this section, we will prove Theorem \ref{thm2}. The main idea is to show \eqref{a04} holds from the assumption \eqref{a05}.
\begin{lem}\label{lem2}
Let $(u, p)$ be a solution of three dimensional Navier-Stokes equations \eqref{a01} with \eqref{a02} and \eqref{a03}. Let $\omega$ be the vorticity. Suppose that $\omega$ satisfies 
\begin{align}\label{d01}
|\omega(x)| \leq \frac{C}{(1+|x'|)^{\beta}},\ 1<\beta<2,
\end{align}
uniformly for $z$. Then $u$ satisfies
\begin{align}\label{d02}
|u(x)| \leq \frac{C}{(1+|x'|)^{\beta-1}},
\end{align}
uniformly for $z$.
\end{lem}
\begin{rem}
We remark that in Lemma \ref{lem2}, there is no restriction that solutions should be axially symmetric.
\end{rem}
\begin{proof}
We know from $\nabla \cdot u=0$ and $-\Delta u=\nabla \times (\nabla \times u)-\nabla (\nabla \cdot u)$ that
\[
-\Delta u=\nabla \times \omega.
\]
Let $\phi$ be the cut-off function defined by \eqref{b03}. A similar calculation as in the proof of Lemma \ref{lem1}
shows that
\begin{equation}\label{d03}
    \begin{split}
\phi u (x) =& \int \Gamma (x-y) \nabla_{y} \times \omega (y) \phi (y) dy\\
            &- \int 2\Gamma (x-y) \nabla_y \phi (y) \cdot \nabla_y u (y) dy
               -\int \Gamma (x-y) \Delta_y \phi(y) u (y) dy\\
           =&-\int \nabla_{y} \Gamma (x-y)  \times \omega (y) \phi (y) dy 
           -\Gamma (x-y)   \omega (y) \times\nabla_{y}\phi (y) dy\\
           &+\int \Gamma (x-y) \Delta_y \phi(y) u (y) dy+2\int \nabla_y \Gamma (x-y) \cdot \nabla_y \phi (y)  u (y) dy.
   \end{split}
\end{equation}
Let $R\rightarrow \infty$ in both sides of \eqref{d03}, we derive
\begin{align}\label{d04}
 u (x) = -\int  \nabla_{y}\Gamma (x-y) \times \omega (y) dy.
\end{align}

Now we will use \eqref{d04} to prove the decay of $u$ under the assumption \eqref{d01}.
The method we use here is similar with \cite{KTW}. The difference is that we will first integrate over $y_3$ and get an estimate for $u$ in terms of an integral over $y'$ instead of $y$. In fact, according to \eqref{d04}, we know that
\begin{equation}\label{d05}
\begin{split}
|u(x)| &\lesssim \int_{\mathbb{R}^3} \frac{1}{|x-y|^2}|\omega(y)|dy\\
         &\lesssim \int_{\mathbb{R}^3} \frac{1}{|x-y|^2}\frac{1}{(1+|y'|)^{\beta}}dy\\
         &\lesssim \int_{\mathbb{R}^2}\int_{-\infty}^{\infty} \frac{1}{|x'-y'|^2+(x_3-y_3)^2}dy_3\frac{1}{(1+|y'|)^{\beta}}dy'\\
         &\lesssim \int_{\mathbb{R}^2}\frac{1}{|x'-y'|}\frac{1}{(1+|y'|)^{\beta}}dy'.
\end{split}
\end{equation}
We decompose the right hand side of \eqref{d05} into three parts as follows.
\begin{align*}
|u(x)| &\lesssim \int_{\mathbb{R}^2}\frac{1}{|x'-y'|}\frac{1}{(1+|y'|)^{\beta}}dy'\\
         &\lesssim \left(\int_{|x'-y'|<\frac{|x'|}{2}}+ \int_{\frac{|x'|}{2}<|x'-y'|<3|x'|}+\int_{3|x'|<|x'-y'|}\right)
         \frac{1}{|x'-y'|}\frac{1}{(1+|y'|)^{\beta}}dy'\\
         &=I_1+I_2+I_3.
\end{align*}
Let's deal with $I_i(i=1,2,3)$ one by one.

For $I_1$, $|x'-y'|<\frac{|x'|}{2}$ implies $\frac{|x'|}{2}<|y'|<\frac{3|x'|}{2}$, i.e., $|y'|\approx |x'|$. Thus 
\begin{align*}
I_1 & \lesssim \int_{|x'-y'|<\frac{|x'|}{2}}  \frac{1}{|x'-y'|}\frac{1}{(1+|y'|)^{\beta}}dy'\\
      & \lesssim \frac{1}{(1+|x'|)^{\beta}}\int_{|x'-y'|<\frac{|x'|}{2}} \frac{1}{|x'-y'|}dy'\\
      & \lesssim \frac{1}{(1+|x'|)^{\beta}}\int_0^{\frac{|x'|}{2}} \frac{1}{r}rdr\\
      & \lesssim \frac{1}{(1+|x'|)^{\beta-1}}.
\end{align*}
For $I_2$, $\frac{|x'|}{2}<|x'-y'|<3|x'|$ implies $|y'|<4|x'|$. Thus
\begin{align*}
I_2 & \lesssim \int_{\frac{|x'|}{2}<|x'-y'|<3|x'|}\frac{1}{|x'-y'|}\frac{1}{(1+|y'|)^{\beta}}dy'\\
      & \lesssim \frac{1}{|x'|}\int_{|y'|<4|x'|} \frac{1}{(1+|y'|)^{\beta}}dy'\\
      & \lesssim \frac{1}{|x'|}\int_0^{4|x|} (1+r)^{-\beta}rdr\ (\text{need}\ \beta<2)\\
      & \lesssim \frac{1}{(1+|x'|)^{\beta-1}}.
\end{align*}
For $I_3$, $|x'-y'|>3|x'|$ implies $|y'|>2|x'|$ and $\frac{|y'|}{2}<|x'-y'|<\frac{3|y'|}{2}$, i.e., $|x'-y'|\approx |y'|$. Thus
\begin{align*}
I_3 & \lesssim \int_{|x'-y'|>3|x'|} \frac{1}{|x'-y'|}\frac{1}{(1+|y'|)^{\beta}}dy'\\
      & \lesssim  \int_{|y'|>2|x'|}\frac{1}{|y'|}\frac{1}{(1+|y'|)^{\beta}} dy\\
      & \lesssim \int_{2|x'|}^{\infty} (1+r)^{-\beta-1}rdr\ (\text{need}\ \beta>1)\\
      & \lesssim \frac{1}{(1+|x'|)^{\beta-1}}.
\end{align*}
By the estimate of $I_i (i=1,2,3)$, we get the desired estimate \eqref{d02}. 
\end{proof}
Now we are ready to prove Theorem \ref{thm2}.
\begin{proof}[Proof of Theorem \ref{thm2}]
It follows from Lemma \ref{lem2} and \eqref{a05} that \eqref{a04} holds. Then by Theorem \ref{thm1}, we get the conclusion in Theorem \ref{thm2}.
\end{proof}

\section*{Acknowledgement}
The author wishes to thank her advisor Prof. Zhen Lei for his support, thank her co-advisor Prof. Qi S. Zhang for sharing ideas, thank Dr. Zijin Li for helpful communications, thank Dr. Wenqi Lyu for finding an error in previous version. The author also wishes to thank Prof. Yat Sun Poon for approving her visit to UC, Riverside and thank China Scholarship Council for support.






\end{document}